\documentclass[12pt]{amsart}
\usepackage{amssymb}
\usepackage[all]{xy}
\usepackage{enumerate}

\usepackage{amsfonts}
\usepackage{amsmath}
\usepackage{amsthm}
\usepackage{amssymb}
\usepackage[all]{xy}
\usepackage[latin1]{inputenc}
\usepackage{graphicx}
\usepackage{color}

\input xy
\xyoption{all}

\newtheorem{theorem}{Theorem}[section]
\newtheorem{corollary}[theorem]{Corollary}
\newtheorem{lemma}[theorem]{Lemma}

\newtheorem{algorithm}[theorem]{Algorithm}
\theoremstyle{definition}
\newtheorem{definition}[theorem]{Definition}
\newtheorem{remark}[theorem]{Remark}

\newtheorem{example}[theorem]{Example}

\newcommand{\F}{\mathbb{F}}
\renewcommand{\Im}{\operatorname{Im}}

\newcommand{\C}{\mathcal{C}}
\newcommand{\Pmat}{\mathbb{P}}
\newcommand{\Amat}{\mathbb{A}}

\newcommand\ord{\operatorname {order}}

\newcommand\df{\operatorname{d}_{\operatorname{free}}}

\title[On $1$-dimensional MDS CGC]{On the construction of 1-dimensional MDS convolutional Goppa codes}
\author[Jos\'{e} I. Iglesias-Curto et al.]{Jos\'{e} I. Iglesias-Curto,
    Francisco J. Plaza-Mart\'{\i}n, and
    Gloria Serrano-Sotelo}
\thanks{University of Salamanca, Deparment of Mathematics and IUFFyM, Plaza de la Merced 1, 37008 Salamanca, SPAIN}
\thanks{This research has been carried out with the financial support of the
Spanish Ministry of Science and Innovation for the project
MTM2012-32342.}

\begin{document}

 \maketitle

\begin{abstract}
We show that the free distance, as a function on a space
parameterizing a family of convolutional codes, is a lower-semicontinuous function and that, therefore, the property of being Maximum  Distance Separable (MDS) is an open condition. For a class of convolutional codes, an algorithm is offered to compute the free distance. The behaviour of the free distance by enlargements of the alphabet and by increasing the length is also studied. As an application, the algebraic equations characterizing the subfamily of MDS codes is explicitly computed for families of $1$-dimensional convolutional Goppa codes (CGC).
\end{abstract}


\section{Introduction}\label{sec:intro}

When constructing convolutional codes, two basic requirements are demanded: usability and high error-correction capability. With regard to the first feature, note that the smaller the alphabet (i.e. the base field), the easier the implementation. With respect to the second
one,  the so-called Maximum  Distance Separable (MDS) codes are the most significant (\cite[Theorem~3.3]{JZ:99}), since they have the largest possible distance between codewords (\cite{RS:99}). It is possible that both properties might not be optimized simultaneously since, for instance, the existence of certain MDS codes has only been proved over large enough fields (\cite{J:75},\cite[Theorem~2.10]{RS:99}). Indeed, it is hard to find references with explicit general methods for constructing MDS convolutional codes (\cite{GL:06, DMS:11}).

In this paper, we report a detailed study of the free distance. More precisely, we provide an algorithm for computing the free distance, a proof of its lower semicontinuity  and its preservation by enlargements of the alphabet. As applications, we offer examples of families of MDS convolutional Goppa codes (CGC, see \cite{DMS:04,MDIS:06}), where the size of the base field has been kept as small as possible, as well as a method to produce MDS CGC of greater length.

Our techniques are based on algebraic tools, as in the pioneering
work of Forney (\cite{For:70}). Similar approaches have been  used fruitfully
in the study of convolutional codes \cite{Lom:01,MIC:10,RR:94} and have
provided  good insight into their structure, such as a
generalization of the Singleton bound (\cite[Theorem~2.2]{RS:99}).

The paper is organized and its main results are presented as follows. After some preliminaries (\S\ref{sec:ConvCodesClass}), a case study is given in \S\ref{subsec:casestu}. This \emph{toy model} clearly exhibits  the type of problems we are concerned with.

The computation of the free distance, $\df$, is a hard
task for arbitrary codes. In order to estimate or compute $\df$,  in some cases it is possible to use the well-known fact that the sequence of row distances converges to $\df$ (\cite[Theorem~3.5]{JZ:99}). Nevertheless, the explicit computation normally requires \emph{ad-hoc} methods for each particular situation. The relation between the sequences of row distances, column distances and active distances and $\df$ is a highly interesting problem when designing convolutional codes (recall for instance the notions of Maximum Distance Profile and strongly MDS, see \cite{HRS:05} and \cite{GRS:03}). In this paper,   for the first time  we are able
to bound the stage at which the sequence of row distances has reached the free distance (Theorem~\ref{th:l(G)}).

The second issue in this section consists of showing that the free distance is preserved when the alphabet is enlarged (Theorem~\ref{thm:FreeDistBaseChange}).

\S\ref{sec:freedist} continues with the study of convolutional codes depending on parameters; i.e., families of codes defined over a parameter space. In this situation, we prove that the free distance, which can be understood as a function from the parameter space to ${\mathbb Z}$, is lower semicontinuous (Theorem~\ref{thm:dfreelowersemicont}). Accordingly, the subset of the parameter space corresponding to MDS codes is open (Corollary~\ref{cor:MDSopen}) or, tantamount to this, the closed subset of non-MDS codes is defined by finitely many algebraic relations in the parameters. This result improves that of~\cite[\S5]{RS:99}, which claims that the subset of MDS convolutional codes contains an open subset.

Applications of these results are to be found in \S\ref{sec:familiesMDSCGC}. First, for the case of $1$-dimensional CGC (\cite{DMS:11}) we offer systematic constructions of families as well as the explicit equations characterizing the locus of non-MDS codes as a subset of the parameter space. By demonstrating examples of MDS codes over small fields, we improve the result of \cite{RS:99} concerning the existence of MDS codes (Corollary~\ref{cor:boundN}), which requires that the field must have sufficiently many elements. Second, we offer a procedure that enables us to increase the length of a CGC, preserving the condition of being MDS (\S\ref{subsec:increasinglength}).

The paper ends with some conclusions and possible directions for future work (\S\ref{sec:conclusion}).



\section{Preliminaries}\label{sec:ConvCodesClass}
\subsection{Convolutional Codes}\label{subsec:convcodes}

We set an arbitrary $q$-ary alphabet where $q$ is a power of a prime number $p$. That is, we work on a finite field, $\F$, with $q$ elements and characteristic $p$. Let us now recall some basic facts on
convolutional codes following the classical references \cite{For:70,Pir:88,McE:98,JZ:99}.

\emph{A convolutional code} is defined to be a $\F(z)$-subspace of $\F(z)^n$. The \emph{length} of the code is the number $n$ and the \emph{dimension} of the code is its dimension as a $\F(z)$-vector space.

However, let us briefly comment the approach in terms of $\F[z]$-modules, which will be more suitable for our study. Note that each $\F[z]$-submodule of $\F[z]^n$, $C$, canonically yields a convolutional code, $C\otimes_{\F[z]}\F(z)$. Since every convolutional code arises in this way, there is a bijective correspondence between convolutional codes and the equivalence classes of $\F[z]$-submodules, where $C$ and $C'$ are defined to be equivalent if $C\otimes_{\F[z]}\F(z) = C'\otimes_{\F[z]}\F(z)$ as subspaces of $\F(z)^n$.

Observe that for $G(z)$ a $k\times n$-matrix with entries in $\F[z]$ we may consider $\phi$ to be the $\F[z]^n$-linear map defined by it:
\begin{equation}\label{eq:ConvCodeSubmod}
     \phi:\F[z]^k\hookrightarrow \F[z]^n
\end{equation}
as well as its image, $C:=\Im\phi$, which is the $\F[z]$-submodule generated by the rows of $G(z)$. In this setup, $G(z)$ is called the \emph{generator matrix} of the convolutional code defined by $C$. It is known that each convolutional code has a generator matrix whose entries are polynomials ({\cite{For:70}).

Recall from commutative algebra that the following three conditions are equivalent: i) the maximal minors of $G(z)$ are coprime (as polynomials in $z$); ii) $\operatorname{Coker}\phi$ is locally free; and, iii) $C$ is a direct summand of $\F[z]^n$. If condition i) is satisfied, we say that $G(z)$ is a \emph{basic generator matrix} for (the convolutional code associated to) $C$. If it is not possible to
reduce the row degrees of the generator matrix by elemental row operations, the matrix is called \emph{reduced}. If it is both basic and reduced, we say that it is \emph{canonical}.

Furthermore, for each convolutional code there exists a representative of the associated equivalence class, $C\subset \F[z]^n$, such that $C$ is a direct summand of $\F[z]^n$. In other words,  every convolutional code admits a basic generator matrix. Consequently, every convolutional code admits a canonical generator matrix (\cite{McE:98}).

Henceforth, and for the sake of brevity, instead of ``the convolutional code corresponding to the equivalence class of a $\F[z]$-submodule $C$'', we shall simply say ``the convolutional code  defined by $C$''.

The row degrees of a canonical generator matrix are invariants of the code
(up to their order) and are known as the \emph{Forney indices} of
the code, $\nu_1,\ldots,\nu_k$. The maximum of the Forney indices
is called the \emph{memory} of the code, $m$, and their sum
is the \emph{degree} (or complexity) of the code, $\delta$, and this
coincides with the highest degree of the maximal minors of the
matrix.

Similarly to the case of block codes, there is a notion of distance that determines the error-correction capability of the convolutional code
(e.g.~\cite[Theorem~3.3]{JZ:99}). Let us first recall that the weight of a polynomial vector is the
number of the non-zero coefficients; that is,
 for a word $c=(a_{10}
    +\ldots+a_{1r_1}z^{r_1},
    \ldots,
    a_{n0}
    +\ldots+a_{nr_n}z^{r_n})$, its weight is:
$$
    \operatorname{w}
    (c)
    \, :=\, \# \{a_{ij}| a_{ij}\neq 0\} \,.
$$

Thus, the free distance between two codewords is the weight of their
difference, and the \emph{free distance} of the convolutional
code, $\df$, is the minimum free distance between two
different codewords. It is therefore natural to look
for codes with the largest possible free distance. Accordingly, Rosenthal and Smarandache (\cite{RS:99}) studied the free distance of an arbitrary convolutional code and proved that it satisfies:
\begin{equation}\label{eq:SingletonBoundConvo}
   \df \leq (n-k)\left(\left\lfloor \frac{\delta}{k} \right\rfloor +1\right)+\delta+1
\end{equation}
This bound is called the \emph{generalized Singleton bound}, since
for codes of degree 0, i.e. block codes, it gives the classical
Singleton bound on the minimum distance. Consequently, \emph{MDS
convolutional codes} are defined as those whose free distance
achieves the generalized Singleton bound.
Note that for the case of $1$-dimensional codes, the generalized Singleton bound acquires
 a  simpler form:
\begin{equation}\label{eq:singletonk=1}
    \df \leq n(\delta+1) \quad .
\end{equation}

\subsection{Classification of Convolutional Codes}
In \cite{MIC:10}, a classifying space for convolutional codes was introduced, providing valuable information about the structure
of convolutional codes. Before  briefly recalling a couple of results of that paper, let us first introduce some definitions which have no counterpart in the theory of block codes.

The $i$-th \emph{column index} of a code $C$ is the maximum of the degrees of the entries of the $i$-th column of a matrix $G(z)$ as $G(z)$ varies among the set of canonical generator matrices for $C$. Let $\{n_i\}_{i=1}^n$ denote the sequence of column indices. It is known (\cite[Theorem
4.9]{MIC:10}) that there is an injective morphism:
$$
  \begin{array}{c}
  \left \{
    \begin{gathered}
          \text{convolutional codes of type $[n,k,\delta;m]$}\\
          \text{with column indices $\{n_i\}_{i=1}^n$}
    \end{gathered}
    \right \} \;
    \hookrightarrow \;
    Gr(\kappa,\mu),
  \end{array}
$$
where $\kappa=k(m+1)-\delta$, $\mu=\sum_{i=1}^n(n_i+1)$ and $Gr(\kappa,\mu)$ denotes the Grassmannian variety of $\kappa$-dimensional subspaces of a given $\mu$-dimensional vector space.

\begin{theorem}[{\cite[Theorem 4.13]{MIC:10}}]\label{thm:caract}
Convolutional codes of type $[n,k,\delta;m]$ that have different
Forney indices, i.e. $\delta<km$, are represented by an
open subset of a closed subset of the Grassmannian
$Gr(\kappa,\mu)$.

Convolutional codes of type $[n,k,\delta;m]$ which have all their
Forney indices equal, i.e. $\delta=km$, are represented by an
open subset of the Grassmannian $Gr(\kappa,\mu)$.
\end{theorem}

%

\subsection{Zariski Topology}\label{subsec:Zariski}

Since this paper deals with the study of some algebraic properties as certain parameters vary, it is natural to consider the Zariski topology. Indeed, \emph{open} and \emph{closed} in the statement of Theorem~\ref{thm:caract} refer to the Zariski topology. Although
this topology can be introduced for very general spaces (e.g.
schemes), the case of the affine space will suffice for our
purposes (\cite{AM:69}). The reason for this assumption relies on the two results above and on the fact that Grassmannian varieties are covered by affine spaces.

A subset $Z$ of the affine space $\F^{r+1}$ is called
\emph{Zariski closed} if it is defined by a finite number of
algebraic relations; that is, there exist polynomials
$p_i(\lambda_0,\ldots,\lambda_r)\in \F[\lambda_0,\ldots,\lambda_r]$ for $i=1,\ldots,s$ such that $Z$
consists of the common zeroes of $p_1,\ldots, p_s$:
    {\small $$
    Z= \big\{ (\alpha_0,\ldots,\alpha_r)\in\F^{r+1} \vert 
    p_i(\alpha_0,\ldots,\alpha_r)=0, 
    1\leq i\leq s\big\}
    $$}
Accordingly, a subset $U\subseteq \F^{r+1}$ is called
\emph{Zariski open} if and only if its complement, $\F^{r+1}\setminus U$, is
Zariski closed.

Moreover, for any field extension $\F\hookrightarrow \F'$ there is a bijective correspondence between the set of $\F'$-valued points of $Z$:
    {\small $$
    Z(\F'):= \big\{ (\alpha_0,\ldots,\alpha_r)\in (\F')^{r+1} \vert
    p_i(\alpha_0,\ldots,\alpha_r)=0,
    1\leq i\leq s\big\}
    $$}
and the set of maps of $\F$-algebras:
    $$
    \F[\lambda_0,\ldots,\lambda_r]/(p_1,\ldots,p_s) \,\to \, \F' \, .
    $$
More explicitly, the map associated with $(\alpha_0,\ldots,\alpha_r)$ sends $\lambda_i$ to $\alpha_i$ for all $i$.

Similarly, one defines the set of $\F'$-valued points of an open subset $U$. It will be denoted by $U(\F')$.

A well known criterion for a subset to be closed, which will be used in the proof of Theorem~\ref{thm:dfreelowersemicont}, is that a subset $Z\subseteq \F^{r+1}$ is closed if and only if for every local and integral $\F$-algebra $A$ and every morphism $\F[\lambda_0,\ldots,\lambda_r]\to A$ the following condition holds: if the composition $\F[\lambda_0,\ldots,\lambda_r]\to A\to A_{(0)}$ defines a point of $Z$, then the composition $\F[\lambda_0,\ldots,\lambda_r]\to A\to A/{\mathfrak m}$ also defines a point of $Z$. Here, $A_{(0)}$ denotes the function field of $A$ and ${\mathfrak m}$ the maximal ideal of $A$.

\section{The free distance}\label{sec:freedist}

Let us begin this section with an example of a convolutional code depending on a parameter. The study of its structure unveils some of its most relevant properties. Indeed, these properties do hold for general codes, as  will be proved rigorously in the following subsections.

\subsection{A Case Study}\label{subsec:casestu}

Let us work on the $2^r$-ary alphabet; that is, on the base field $\F_{2^r}$. We consider the matrix:
    {\small $$
    \left(
    \begin{array}{cccc}
    \lambda + z & \lambda + 1 + z  & \lambda z & 1 + (\lambda + 1) z  \\
    \lambda^2+(\lambda+1)z & 1+z & \lambda+(\lambda+1)z & (\lambda+1)^2+\lambda z
    \end{array}
    \right)\; ,
    $$}
and let $C_{\lambda}$ be its image. After some computation, we see that this matrix is a basic generator matrix of $C_{\lambda}$, for $\lambda$ an arbitrary element of $\F_{2^r}\setminus \{0\}$. In this case, $C_{\lambda}$ has length $n=4$, rank $k=2$, degree  $2$,  memory  $1$  and, consequently, the generalized Singleton bound is $7$ (see equation~(\ref{eq:SingletonBoundConvo})).

Let us choose a particular case; namely, let $r=3$ and let $\lambda\in \F_8=\F_{2^3}$ be an element satisfying $\lambda^3+\lambda+1=0$. After some computations, one has that the sequence of row distances of $C_{\lambda}$ (see~\cite[Chapter~3]{JZ:99}) is the constant sequence $7,7,7,\ldots$, such that the free distance of $C_{\lambda}$ is $7$, and, therefore, it is MDS.

Let us now discuss the general situation; i.e. $r\geq 3$ and $\lambda$ arbitrary. First, let us express the above generator matrix as $G(z)=G_0 +
zG_1$, with $G_0,G_1$ matrices over $\F_{2^r}[\lambda]$. In this example, we observe that the determinant of ${G_0 \choose G_1}$ appears in the list of $2\times2$-minors of $(G_0\,|\,G_1)$ and, therefore, it can be
checked that $C_{\lambda}$ is MDS if and only if  the matrix $(G_0\,|\,G_1)$
generates an MDS block code; i.e.  all the $2\times2$-minors of $(G_0\,|\,G_1)$ are non-zero.   Summing up, $C_{\lambda}$ fails to be MDS if $\lambda$ satisfies any of the following equations:
    $$
    \begin{array}{l}
       \lambda + 1 = 0\\
       \lambda^2+\lambda+1 = 0\\
       \lambda^3+\lambda^2+1=0
    \end{array}
    $$

We also observe that the number of values of $\lambda$ for which $C_{\lambda}$ fails to be MDS is finite.

Summing up, we have observed the following three phenomena:
\begin{enumerate}
    \item Although we know that the sequence of row distances converges to the free distance, there is no estimation for the stage in which the sequence reaches $\df(C_{\lambda})$;
    \item If $C_{\lambda}$ is MDS for $\lambda\in\F_{2^r}$ and $\F_{2^r}\hookrightarrow\F'$, then the code generated by $C_{\lambda}$ over $\F'$ is also MDS; or, in simpler words, the property of being MDS depends only on $\lambda$ and not on the base field;
    \item In a family of codes depending on parameters, the subset of the parameter space consisting of those values for which the code fails to be MDS is defined by a finite number of algebraic relations.
    \end{enumerate}

The three following subsections are devoted to an in-depth study of these issues.

%
%
%
%
%
%

\subsection{Computation of the free distance}\label{subsec:l(C)}

Let us consider a convolutional code $C$ of dimension
$k$, length $n$ and degree $\delta$, with a canonical generator matrix $G(z)$ decomposing as:
    $$
    G(z)=G_0+G_1z+\ldots+G_\delta z^\delta
    $$
where $G_i$ are $k\times n$ matrices with entries in $\F$. In this case, $G_0$ has maximal rank and, therefore, the linear map defined by the $k(l+1)\times
n(\delta+l+1) $-matrix:
    \begin{equation}\label{eq:Gk=slidding}
    {\small\begin{pmatrix}
        G_{0} & \dots & G_{\delta} & 0 & \dots &  0 \\
        0 & G_{0} & \dots & G_{\delta} & \ddots  &  \vdots
        \\
        \vdots &  \ddots & \ddots & & \ddots  &  0 \\
        0 & \dots & 0 & G_{0} &\dots &   G_{\delta}
        \end{pmatrix}}\;
        :\F^{k(l+1)}\to \F^{n(\delta+l+1)}
        \, ,
    \end{equation}
is injective for $l\geq 0$. Let $C_l$ be the linear code defined by the image of this map and let $d^r_l$ denote the distance of  $C_l$. In \cite[Chapter~3]{JZ:99}, $d^r_l$ is called the $l$-th \emph{row distances}. Since any basic encoder is non-catastrophic (\cite{MS:68}), it follows from
\cite[Chapter~3]{JZ:99} that starting with a basic encoder $G(z)$, the sequence $\{d^r_l\}$ is non-increasing and  eventually
converges to the free distance of the convolutional code; that is:
    \begin{equation}\label{eq:dfree=drl}
    \df(C) \,=\,
    \underset{l\geq 0}{\operatorname{min}} \{d^r_l\}\, .
    \end{equation}

%
%
%

However, an explicit description of the step in which the sequence of row
distances reaches the free distance has not been given. Our next
result tackles this problem.

\begin{theorem}\label{th:l(G)}
Let $C$ be a convolutional code as above.
Let $C^0$ be the linear code defined by the image of the map:
    \begin{equation}
    \begin{pmatrix}
            G_{\delta} \\
            G_{\delta-1}  \\
            \vdots  \\
            G_{0}
            \end{pmatrix}\colon
            \; {\mathbb F}^{k(\delta+1)} \,\to\, {\mathbb F}^{n}
    \end{equation}
and let $\nu$ and $\mu$ be the dimension and the distance of $C^0$.

If $\nu$ is maximal, i.e. $k(\delta+1)= \nu$, then:
    $$
    \df(C) \, =\, d^r_{\operatorname{l}(C)}
    $$
where: \\
\begin{equation}\label{eq:lambdaG=integerpart}
        \operatorname{l}(C) \,:= \, 
    \left\lfloor \frac{1}{\mu}
        \Big(
            (n-k)\big(\left\lfloor \frac{\delta}{k} \right\rfloor +1\big)+\delta+1
        \Big)\right\rfloor
        \; -(\delta +1)
\end{equation}
\end{theorem}

\begin{proof}
Let us consider a non-zero
polynomial vector of degree $l\geq \delta$, $\alpha(z)=\sum_{i=0}^l \alpha_i z^i \in \F[z]^k$, and the codeword given by it:
    {\small
    \begin{equation}\label{eq:alpha(z).G=sum}
    \begin{aligned}
    \alpha(z)& G(z)\,=\,
    (\alpha_0,\ldots,\alpha_l)
        \begin{pmatrix} G_{0} \\ 0 \\ \vdots  \\ 0 \end{pmatrix}
        \,+\,
    (\alpha_0,\ldots,\alpha_l)
        \begin{pmatrix} G_1 \\ G_0 \\ 0 \\ \vdots  \\ 0 \end{pmatrix} z
        \,+\, \ldots \\
        &
        \,+\,
    (\alpha_0,\ldots,\alpha_l)
        \begin{pmatrix} G_\delta \\ \vdots \\ G_0 \\ 0 \\ \vdots  \\ 0 \end{pmatrix} z^{\delta}
        \,+\, \ldots
         + \,
    (\alpha_0,\ldots,\alpha_l)
        \begin{pmatrix} 0 \\ \vdots \\ 0 \\ G_{\delta} \\  \vdots \\ G_0 \end{pmatrix}
        z^{l} \,+
        \\
        &
        \,+\, \ldots \,+\,
    (\alpha_0,\ldots,\alpha_l)
        \begin{pmatrix} 0 \\ \vdots \\ 0  \\ G_{\delta} \end{pmatrix} z^{\delta+l}
    \end{aligned}
    \end{equation}
    }
Our task consists of finding a lower bound for the weight of this
codeword.

Suppose that there exists $j$, with $\delta\leq j\leq l$,  such that the coefficient of $z^j$ vanishes. The hypothesis $k(\delta +1)= \nu$ implies that $\alpha_{j-\delta}=\alpha_{j-\delta+1}=\ldots=\alpha_{j}=0$. In that case we would have:
    {\small $$
    (\sum_{i=0}^l \alpha_i  z^i) G(z)
        \,=\,
    (\sum_{i=0}^{j-\delta-1} \alpha_i z^i) G(z)
        \,+\,
    z^{j+1}(\sum_{i=j+1}^{l} \alpha_{i} z^{i-(j+1)}) G(z)
    \, .
    $$}
Noting that the coefficients of the second term on the r.h.s.
cannot cancel the coefficients of the first term, one has that:
\begin{multline*}
    \operatorname{w}\big((\sum_{i=0}^l \alpha_i  z^i)\cdot G(z)\big)
        \,\geq \,
    \operatorname{w}\big((\sum_{i=0}^{j-\delta-1} \alpha_i z^i)\cdot G(z)\big)
    \,\geq \\
    \,\geq \, d^r_{j-\delta-1} \,\geq \, d^r_l \,\geq \,
    \df(C)
\end{multline*}

Therefore, when bounding the minimum weight of codewords from
below we can leave aside those words with
$\alpha_{j-\delta}=\alpha_{j-\delta+1}=\ldots=\alpha_{j}=0 $ for
certain $\delta\leq j\leq l$. Let us define:
    {$$m_l\,:=\,
    \operatorname{min}
    \big\{
    \operatorname{w}\big((\sum_{i=0}^l \alpha_i  z^i) G(z)\big)
    :
    \alpha_0\neq 0 , \alpha_l\neq 0
    \big\}
    $$}
 for $l < \delta$, and:
    {\small $$
    m_l\, :=\, \operatorname{min}
    \left\{
    \begin{gathered}
    \operatorname{w}\big((\sum_{i=0}^l \alpha_i  z^i) G(z)\big)
    \text{ such that } \alpha_0\neq 0 , \alpha_l\neq 0
    \\
    \text{and } \nexists j\in\{\delta,\ldots, l\} \text{ with }
    \alpha_{j-\delta}=
    \ldots=\alpha_{j}=0
    \end{gathered}
    \right\}
    \, .
    $$}
for $l\geq \delta$.

The previous discussion shows that:
    \begin{equation}\label{eq:G_lambda=minm_k}
    d^r_l\,=\,
    \operatorname{d}(C_l) \,=\,
    \operatorname{min}\{m_0,\ldots, m_l\}
    \end{equation}

Let us bound $m_l$ from below. First, let us consider $l<\delta$. The hypothesis $k(\delta+1)= \nu$, shows that $k(j+1)= \operatorname{rk}{\tiny \begin{pmatrix}  \\ G_{j} \\  \vdots \\ G_0 \end{pmatrix}}$ for all $j\leq l$. In particular, if $\alpha_0\neq 0$ then the coefficients of $z^0,z,\ldots, z^\delta$ cannot vanish. Similarly, $\alpha_l\neq0$ implies that the coefficients of $z^l,\ldots,z^{\delta+l}$ do not vanish either. Consequently,
	$$
	m_l\,\geq \,  (l+\delta +1)\mu \qquad \forall l <\delta
	$$
Second, let $l\geq \delta$. The above arguments imply that the coefficients of $z^0,z,\ldots, z^{l+\delta}$ do not vanish and, thus:
	$$
	m_l\,\geq \,  (l+\delta +1)\mu \qquad \forall l \geq \delta
	$$
	
Having in mind equations~(\ref{eq:dfree=drl}) and~(\ref{eq:G_lambda=minm_k}), we know that there exists $l_0$ such that
$\df(C)=m_{l_0}$. Recalling the
Singleton bound for $C$
(equation~(\ref{eq:SingletonBoundConvo})), we derive the following
chain of inequalities:
	{\small $$
    (n-k)\big(\left\lfloor \frac{\delta}{k} \right\rfloor +1\big)+\delta+1
    \,\geq \,
    \df(C)
    \,=\,      m_{l_0}
    \, \geq \, (l_0+\delta+1) \mu
$$}
and thus:
    $$
    l_0\,\leq \,
    \frac{1}{\mu}
        \Big(
            (n-k)\big(\left\lfloor \frac{\delta}{k} \right\rfloor +1\big)+\delta+1
        \Big)
        \; -(\delta +1)
    $$
and the conclusion follows.
\end{proof}

It is worth pointing out that the previous result provides a method to compute the free distance of a convolutional
code in terms of the distance of a linear code. More precisely, we have the following:
\begin{algorithm}
For a convolutional code with a canonical generator matrix
$G(z)$, its free distance is given by the following procedure:
    \begin{enumerate}
	\item compute the dimension and the distance of $C^0$; that is,  $\nu$ and $\mu$ respectively;
    \item  check if $k(\delta+1)= \nu$; if yes, continue;
    \item compute $\operatorname{l}(C)$ by
    equation~(\ref{eq:lambdaG=integerpart});
    \item  $\df(C)$ is given by the distance of the linear code
    $C_{\operatorname{l}(C)}$ (see
    equation~(\ref{eq:Gk=slidding})).
    \end{enumerate}
\end{algorithm}

\subsection{Preservation of $\df$ by enlargements of the alphabet}\label{subsec:enlargealphabet}

Similarly to the relation between BCH and RS codes, we begin by studying whether the free distance changes when the alphabet is enlarged or, equivalently, when the base field $\F$ is replaced by a finite extension of it, say $\F'$. Indeed, if $G(z)$ is a matrix with entries in $\F[z]$ generating a convolutional code $C$, then the words of $C$ are $\F[z]$-linear combinations of the rows of $G(z)$. Hence, for a finite extension $\F\hookrightarrow\F'$, and considering $\F'[z]$-linear combinations of the rows of $G(z)$, we obtain another convolutional code $C'$ which is related to $C$ by the identity $C':= C\otimes_{\F}\F'$.


\begin{theorem}\label{thm:FreeDistBaseChange}
The free distance is preserved by enlargements of the alphabet.
\end{theorem}

\begin{proof}
With the previous notations, the following equation must be proved:
    $$
    \df(C) \,=\, \df(C')
    $$

Since $\F\hookrightarrow\F'$, every word of $C$ can be also understood as a word of $C'$. Thus, $\df(C)\geq \df(C')$.

Let us now check the opposite inequality. Let $G(z)$ be a generator matrix for $C$. It is therefore also a generator matrix for $C'$. Let us choose a non-zero codeword of $C'$ of minimal weight, which will be of the type:
    {\small
    \begin{multline}\label{eq:codewordC'}
    \big(
    \sum_j \alpha_{1,j}z^j , \ldots, \sum_j \alpha_{n,j} z^j\big)
    \,=\\
    \big(
    \sum_j \beta_{1,j}z^j , \ldots, \sum_j \beta_{k,j} z^j\big) \cdot G(z)
    \end{multline}}
where $\beta_{i,j}\in\F'$. Let us consider $\F'':=\F(\{\beta_{i,j}\}_{i,j})$, which is a finite extension of $\F$. Let $\omega:\F''\to \F$ be a $\F$-linear form. Since the following identity holds:
    {\small
    \begin{multline}\label{eq:codewordC}
    \big(
    \sum_j \omega(\alpha_{1,j})z^j , \ldots, \sum_j \omega(\alpha_{n,j}) z^j\big)
    \,=\\
    \big(
    \sum_j \omega(\beta_{1,j})z^j , \ldots, \sum_j \omega( \beta_{k,j}) z^j\big) \cdot G(z)
    \end{multline}}
it follows that its left hand side is a codeword of $C$. Thus, if there exists $\omega$ and $i,j$ such that $\omega(\beta_{i,j})\neq 0$, then it follows that the weight of the codeword~(\ref{eq:codewordC}) is non-zero and equal to or smaller than the weight of the codeword~(\ref{eq:codewordC'}) and, therefore, $\df(C) \,\leq \, \df(C')$. On the other hand, let us assume that $\omega(\beta_{i,j})= 0$ for all $\omega$ and all $i,j$. Since $\F''$ is a finite dimensional $\F$-vector space, it follows that $\beta_{i,j}=0$ for all $i,j$ and hence the codeword~(\ref{eq:codewordC'}) is zero, which contradicts our hypothesis.
\end{proof}

\subsection{Variation of $\df$ along a family}\label{subsec:variationdfree}

Our second result concerning the free distance is related to its variation along the members of a family of convolutional codes. In particular,
we are interested in how the condition of being MDS behaves.

Inspired by the definition of a convolutional code
(see~\S\ref{subsec:convcodes}), let us introduce the notion of a
family of convolutional codes; that is, a convolutional code
depending on parameters $(\lambda_0,\ldots,\lambda_r)$, where
$\lambda_i$ takes values in $\F$.


\begin{definition}\label{def:FamilyConvCode}
A family of convolutional codes of length $n$ and rank $k$ defined over an open subset $U$ of $\F^{r+1}$ consists of a locally free $\F[\lambda_0,\ldots,\lambda_r][z]$-submodule ${\mathcal C}$ of $\F[\lambda_0,\ldots,\lambda_r][z]^n$ such that the restriction of $\F[\lambda_0,\ldots,\lambda_r][z]^n/{\mathcal C}$ to $U$ is locally free of rank $n-k$. $U$ will be called the \emph{parameter space} of the family.
\end{definition}

For the sake of clarity, let us interpret this definition in terms of matrices. First, note that we are replacing our field $\F$ by the $\F$-algebra  $R:=\F[\lambda_0,\ldots,\lambda_r]$ and, in particular, that the elements of $R[z]$ are polynomials in $z$ whose coefficients are polynomials in $\lambda_0,\ldots,\lambda_r$. Now let ${\mathcal G}(z)$ be a $k\times n$-matrix with entries in $R[z]$. Therefore, for each point $(\alpha_0,\ldots,\alpha_r)$ of the affine space $\F^{r+1}$ we evaluate the entries of ${\mathcal G}(z)$ at $\lambda_i=\alpha_i$ for $i=0,\ldots,r$ and obtain a $k\times n$-matrix, say ${\mathcal G}_{\alpha}(z)$, with entries in $\F[z]$. Now, let $Z$ be the subset of points $\alpha\in\F^{r+1}$ such that ${\mathcal G}_{\alpha}(z)$
has rank smaller than $k$. Let $V$ be the set of points $\alpha\in\F^{r+1}$ such that there exists a neighborhood of $\alpha$ in which  $R[z]^n/{\mathcal C}$ is free. Observe that $Z$ is closed and $V$ is open.  Therefore, the submodule ${\Im}({\mathcal G}(z))$ defines a family of convolutional codes of length $n$ and rank $k$ over the open subset $U:= V\setminus Z$.

Recall that, if a generator matrix is given, $V$ consists of those points of the parameter space where the minors of maximal rank of the generator matrix are coprime and, consequently, the generator matrix yields a canonical generator matrix of ${\mathcal C}\vert_V$.

From now on, calligraphic letters (such as ${\mathcal C}, {\mathcal G},\ldots$) will refer to families, while roman letters (i.e. $C$, $G$,\ldots) will be used for the case of a convolutional code over a field.

\begin{remark}
For $r=0$ one has that $R[z]=\F[\lambda_0,z]$ and hence the
entries of the generator matrix of a family over $\F$ are
polynomials in two variables, namely, $\lambda_0$ and $z$.
Therefore, 2D convolutional codes (\cite{FV:94}) are instances of
$1$-parameter families of convolutional codes. Similarly,  convolutional codes depending on several variables might be considered as families in
the sense of Definition~\ref{def:FamilyConvCode}.
\end{remark}

Accordingly,  if ${\mathcal C}$ is a family of convolutional codes over the
parameter space $U\subseteq \F^{r+1}$, we may consider the \emph{map of sets}:
    \begin{equation}\label{eq:dfree}
    \begin{aligned}
    \df: U\, &
    \longrightarrow\, {\mathbb Z}
    \\
    \alpha=(\alpha_0,\ldots,\alpha_r) \,&
        \longmapsto \, \df({\mathcal C}_{\alpha})
    \end{aligned}\, .
    \end{equation}
However, the previous subsection shows that $\df({\mathcal C}_{\alpha})$ is well defined, disregarding whether $(\alpha_0,\ldots,\alpha_r)$ is considered as a point with coordinates in $\F$ or in an extension $\F\hookrightarrow \F'$. Consequently, $\df$ \emph{is really a function} when $U$ is endowed with the Zariski topology; i.e. the topology inherited from $\F^{r+1}$. Before proving an important property of it, let us recall that a function $f:X\to {\mathbb Z}$ is called
lower semi-continuous if and only if:
    $$
    f^{-1}\big( (N,+\infty)\big)
    \, :=\,
    \{x\in X \,\vert\, f(\alpha) >N \}
    $$
is Zariski open for every $N\in {\mathbb Z}$. Here,
$(N,+\infty)\subset {\mathbb Z}$ denotes the open interval.

\begin{theorem}\label{thm:dfreelowersemicont}
Let ${\mathcal C}$ be a family of convolutional codes over the
parameter space $U\subseteq \F^{r+1}$.

The function $\df:U\to {\mathbb Z}$ of equation~(\ref{eq:dfree})
is lower semi-continuous.
\end{theorem}

\begin{proof}
We shall  use the criterion given at the end of \S\ref{subsec:Zariski} in order to prove that the subset:
    $$
    \{ \alpha=(\alpha_0,\ldots,\alpha_r)\in U \,\vert
     \, \df({\mathcal C}_{\alpha}) \leq N
     \}
    $$
is a Zariski closed subset of $U$. Let $A$ be a local and integral $\F$-algebra and let ${\mathfrak m}$ be its maximal ideal. It suffices to consider the case where the family of convolutional codes ${\mathcal C}$ is defined over $A$ (i.e. all entries of the matrices belong to $A$). We must prove that the free distance of the code  ${\mathcal C}_{(0)}:= {\mathcal C}\otimes_A A_{(0)}$ is equal to or greater than that of the code ${\mathcal C}_{{\mathfrak m}}:= {\mathcal C}\otimes_A A/{\mathfrak m}$. Since $A$ is local, it follows that ${\mathcal C}$ is a free submodule of $A[z]^n$ and has a generator matrix ${\mathcal G}(z)$. Observe that the codewords of ${\mathcal C}_{(0)}$ are $A_{(0)}$-linear combinations of the rows of ${\mathcal G}(z)$.

Now let $ c'=(\sum_j \beta'_{1,j}z^j,\ldots, \sum_j \beta'_{k,j}z^j)\cdot{\mathcal G}(z) $ be a non-zero codeword of ${\mathcal C}_{(0)}$ of minimal weight. Since there is a finite number of non-zero coefficients, and since $A_{(0)}$ consists of quotients $\frac{a}{b}$ with $a,b\in A$ and $b\neq 0$, there exists $b\in A$ such that $c:= b\cdot c'  =(\sum_j \beta_{1,j}z^j,\ldots, \sum_j \beta_{k,j}z^j)\cdot{\mathcal G}(z)$, which is another codeword of the same weight, belongs to ${\mathcal C}$. Taking the class of $\beta_{i,j}$ modulo ${\mathfrak m}$, say $\bar\beta_{i,j}\in A/{{\mathfrak m}}$, we obtain a codeword $\bar c \in {\mathcal C}_{{\mathfrak m}}$.

It is now straightforward to see that $\df({\mathcal C}_{(0)})$, which coincides with the weight of $c'$, is equal to or greater than the weight of $\bar c$ and, therefore, equal to or greater than $\df({\mathcal C}_{\mathfrak m})$.
\end{proof}

\begin{remark}
For a family of convolutional codes ${\mathcal C}$ over the parameter space
$U\subseteq \F^{r+1}$  the length and the dimension are constant for all codes of the family.
However, the degree, the memory and the column indices may vary. To illustrate how the degree changes, let us consider the degree  as a
function:
    $$
    \begin{aligned}
    \delta:U\, &
    \longrightarrow\, {\mathbb Z}
    \\
    \alpha=(\alpha_0,\ldots,\alpha_r) \,&
        \longmapsto \, \delta({\mathcal C}_{\alpha}) \,:=\,
        \text{ degree of } {\mathcal C}_{\alpha}
    \end{aligned}
    $$
It follows straightforwardly that $\delta$ is a lower
semi-continuous function. Analogous statements can be
proved for the memory and the column indices of the family.
\end{remark}

In~\cite[\S5]{RS:99} it was shown that the subset of the set of all convolutional codes consisting of those that are MDS contains a Zariski open subset. Our previous theorem now allows us to strengthen that claim.

\begin{corollary}\label{cor:MDSopen}
The subset of MDS
convolutional codes is a Zariski open subset of the set of all convolutional codes.
\end{corollary}

\begin{proof}
Recall from Theorem~\ref{thm:caract} that the set of convolutional
codes, ${\mathfrak C}$, can be understood as a subset of a
Grassmannian and thus it inherits the Zariski topology. Let
${\mathfrak M}$ denote the subset of ${\mathfrak C}$ consisting of
those points corresponding to MDS convolutional codes. Thus,
${\mathfrak M}$ will be an open subset of ${\mathfrak C}$ if and only if
there is a covering of the Grassmannian by Zariski open subsets
$\{V_i\}$ such that ${\mathfrak M}\cap V_i$ is open in ${\mathfrak
C}\cap V_i$ for all $i$.

Bearing in mind that Grassmannians are covered by affine spaces
or, equivalently, that $V_i$ can be assumed to be isomorphic to
$\F^{r+1}$ for some $r$, it suffices to prove the following
statement: let ${\mathcal C}$ be a family of convolutional codes
over the parameter space $U$ where $U$ is a subset of $\F^{r+1}$;
accordingly, the subset $\{\alpha\in U\vert {\mathcal C}_{\alpha}\text{
is MDS}\}$  is Zariski open in $U$.

Furthermore, since we are dealing with a single Grassmaniann of the
type $Gr(\kappa,\mu)$, and hence connected, it may be assumed
that the degree, the memory, and the sum of the column indices are constant
along the family ${\mathcal C}$.
Theorem~\ref{thm:caract} implies that there is a well defined map:
    $$
    \begin{aligned}
    U & \,\overset{\psi}\longrightarrow\, {\mathfrak C}\,\subset\, Gr(\kappa,\mu)
        \\
    (\alpha_0,\ldots,\alpha_r) & \longmapsto \, {\mathcal C}_\alpha
    \end{aligned}
    $$
and, in particular, that the Singleton bound
(\ref{eq:SingletonBoundConvo}) is the same constant, say $N$, for
each code of the family corresponding to points of $U$.

Now, the subset of $U$ consisting of MDS convolutional codes is
precisely:
\begin{align*}
    \{\alpha\in U\;\vert\; {\mathcal C}_{\alpha} \text{ is MDS}\}
    & \,=\,
    U\cap \psi^{-1}({\mathfrak M}) \\
    & \,=\,
    \df^{-1}\big((N-1, +\infty)\big)\cap U
\end{align*}
($\df:U\to {\mathbb Z}$ being as above), which is Zariski open in
$U$ by  Theorem~\ref{thm:dfreelowersemicont}.
\end{proof}

%

\begin{corollary}\label{cor:boundN}
Let ${\mathcal C}$ be a family of convolutional codes over a parameter space $U$, where $U$ is a non-empty open subset of $\F$.

For a finite extension $\F\hookrightarrow \F'$, let us define:
    $$
    N(\F')\,:=\, \#
    {\small
    \left\{
     \begin{gathered}
     \alpha \in U(\F') \text{ such that}
     \\
     ({\mathcal C}\otimes_{\F}\F')_{\alpha}\text{ is non-MDS}
     \end{gathered}
     \right\} }
    $$
where $U(\F')$ is the set of $\F'$-valued points of $U$ (\S\ref{subsec:Zariski}).

If $N(\F)>0$, then there exists $N \in{\mathbb N}$, depending only on $\mathcal C$, such that:
    $$
    N(\F')\,\leq \, N
    $$
for any finite extension $\F\hookrightarrow \F'$.
\end{corollary}

\begin{proof}
Let $\bar\F$ denote the algebraic closure of $\F$. Recall that ${\mathcal C}$ is a submodule of $\F[\lambda_0][z]^n$ and thus $({\mathcal C}\otimes_{\F}\bar\F)$ is a submodule of $\bar\F[\lambda_0][z]^n$. For each $\alpha\in\bar \F$, evaluating $\lambda_0$ at $\alpha$, one obtains a convolutional code $({\mathcal C}\otimes_{\F}\bar\F)_{\alpha}$.

Let $V_1$ be the set of $\alpha\in\bar\F$ such that $({\mathcal C}\otimes_{\F}\bar\F)_{\alpha}$ is locally free of rank $k$. Let $V_2$ be the maximal  subset of $\bar\F$ where $\bar\F[\lambda_0][z]^n / ({\mathcal C}\otimes_{\F}\bar\F)$ is locally free. It is easy to check that both $V_1$ and $V_2$ are open; ${\mathcal C}\otimes_{\F}\bar\F$ defines a family of convolutional codes over $V_1\cap V_2$, and $V_1\cap V_2$ contains $U(\F')$ for any extension $\F'$. Corollary~\ref{cor:MDSopen} shows that:
 {\small $$
 W\;:=\;\{\alpha\in  V_1\cap V_2 \;\vert\;
 ({\mathcal C}\otimes_{\F}\bar\F)_{\alpha} \text{ is MDS}
 \}
 $$}
is an open subset of $\bar\F$, which is non-empty by hypothesis.

Since non-trivial closed subsets of $\bar\F$ are finite sets, it follows that $N(\bar\F)$ is finite and coincides with the number of points of the complement of $W$. Bearing in mind that $N(\F')\leq N(\bar\F)$ for any finite extension $\F\hookrightarrow \F'$, the claim is proved.
\end{proof}


\section{Applications to $1$-dimensional MDS CGC}\label{sec:familiesMDSCGC}

This section applies the previous results for the study of MDS convolutional Goppa codes. In particular, it is worth noticing that these computations are carried out on small fields and  consequently that there is no need to consider fields with sufficiently many elements in order to prove that the set of $1$-dimensional MDS CGC is non-empty. Recall that the non-emptiness of the set of MDS convolutional codes was proved in \cite{RS:99} under the assumption that the field was large enough.

%
%

For simplicity, we shall restrict ourselves to the case of
$1$-dimensional codes constructed on the projective line (\cite{DMS:11}). Interested readers can check the general
construction of  higher-dimensional  CGC  on arbitrary curves in \cite{DMS:04, MDIS:06}. Observe that if a $1$-dimensional convolutional code is MDS, then  its column indices are all equal to the degree. Hence, we also assume that $n_i=\delta$ for all $i\leq n$.

Let us briefly recall the construction of CGC. Let $\Pmat^1$ be the projective line over the field $\F(z)$. Fix
 a point at infinity, $p_\infty$, and an affine coordinate $t$
in $\Amat^1=\Pmat^1 \backslash p_\infty$.

Let us consider  $n$ different points $p_i$ with affine coordinates $\{a_iz+b_i\}_{1\leq i\leq n}$ such that $a_i\neq 0$. Let $G=\delta p_{\infty}$ with $0\leq \delta<n$,  $L(G)=\langle 1,t,\dots,t^{\delta}\rangle$ be the associated linear series and let $s(t)=\sum_{i=0}^{\delta}\lambda_i t^i\in L(G)$ be a rational function.

Now, consider $\bar {\mathcal C}$  the  submodule given by the image of the evaluation of $s(t)$ at the points $p_i$:
    \begin{align*}
    \langle s(t)=\lambda_0+\lambda_1t+\cdots+\lambda_{\delta}t^{\delta}
    \rangle
    &
    \rightarrow
    \F[\lambda_0,\ldots,\lambda_{\delta}][z]^n
    \\
    s(t)&\mapsto (s(p_1),\dots, s(p_n))\, ,
    \end{align*}
as a family of submodules parametrized by $\lambda=(\lambda_0,\ldots,\lambda_{\delta})\in {\mathbb P}^{\delta}$. Let us prove that the restrictions of $\bar{\mathcal C}$ and $\F[\lambda_0,\ldots,\lambda_n][z]^n/\bar{\mathcal C}$ to the open subset:
	{\small
	\begin{equation}\label{eq:parameterspace}
	V\,:=\,
	\bigcap_{1\leq i<j\leq \delta} \left\{\begin{gathered} (\lambda_0,\ldots,\lambda_{\delta}) \text{ such that} \\ \sum_{k=0}^{\delta}\lambda_k
	(b_j-b_i)^k(a_i-a_j)^{\delta-k}\neq 0\end{gathered}\right\}
	\end{equation}	}
are locally free.

Note that a matrix associated to the evaluation map above is given by:
    $$
    {\mathcal G}(z)=
    \begin{pmatrix}
        \displaystyle\sum_{i=0}^{\delta}\lambda_i (a_1z+b_1)^i &\dots &\displaystyle\sum_{i=0}^{\delta}\lambda_i (a_n z+b_n)^i
        \end{pmatrix}
        \, ,
    $$
and that $\F[\lambda_0,\ldots,\lambda_n][z]^n/\bar{\mathcal C}$ fails to be locally free whenever these entries have a common root as polynomials in $z$. Let $\alpha$ be a common root, hence $a_i\alpha+b_i$ is a root of $s(t)=\sum_{k=0}^{\delta} \lambda_k t^k$ for all $i$. Since $s(t)$ has degree $\delta$ and $\delta<n$, there exist $i,j$, with $1\leq i<j\leq n$, such that:
	$$
	a_i\alpha+b_i\, =\, a_j\alpha+b_j
	$$
For these indices, it must hold that $a_i\neq a_j$ since, otherwise, this equation implies that $b_i=b_j$ and, hence, $p_i=p_j$ which contradicts the hypothesis. Observe that $s(\frac{b_j-b_i}{a_i-a_j})=s(\alpha)=0$ implies that $\sum_{k=0}^{\delta}\lambda_k(b_j-b_i)^k(a_i-a_j)^{\delta-k}=0$.

Hence, ${\mathcal C}:=\{{\mathcal C}_{\lambda}\}_{\lambda\in V}$ defines a family of  $1$-dimensional length $n$ CGC with parameter space given by~(\ref{eq:parameterspace}), canonical generator matrix equal to ${\mathcal G}(z)$ and degree given by the degree of $s(t)$.  ${\mathcal G}(z)$ admits a polynomial decomposition as follows:
    $$
    {\mathcal G}(z)=G_0+G_1z+\cdots+G_{\delta}z^{\delta}
    $$
with:
    {\small \begin{equation}\label{eq:caract}
    \begin{aligned}
    G_j
    \,=\, &
    \begin{pmatrix}\displaystyle\sum_{r=j}^{\delta}\lambda_r
    \binom{r}{j}a_1^j b_1^{r-j} &\dots
    &\displaystyle\sum_{r=j}^{\delta}\lambda_r \binom{r}{j}a_n^j
    b_n^{r-j}\end{pmatrix}
    \,= \\ = \,&
    \begin{pmatrix} a_1^j s_{\lambda}^{(j)}(b_1)&\dots& a_n^j s_{\lambda}^{(j)}(b_n)\end{pmatrix}
    \end{aligned}
    \end{equation}}
where, for the sake of notation, we have  introduced  the functions $s^{(j)}: \F^{\delta+1}\times\F\to \F$ defined by:
$$
s_{\lambda}^{(j)}(t)=\sum_{r=j}^{\delta}\binom rj {\lambda}_rt^{r-j}\,,\quad 0\leq j\leq\delta\,,\, {\lambda}=({\lambda}_0,\dots,{\lambda}_{\delta})\,.
$$

\subsection{A Family of MDS CGC}

Let us now consider a family of CGC and let us compute the algebraic equations of the subset of the parameter space corresponding to non-MDS codes  explicitly. These results are indeed consequences of \S\ref{subsec:variationdfree}.

Let us consider $n$ pairwise different points
$\{p_i=a_iz+b\}_{1\leq i\leq n}$ such that $a_i\neq0$ and  the
rational function $s(t)=\sum_{i=0}^{\delta} \lambda_i t^i$ (where $\delta<n$). Observe that, in this situation, the open subset defined by (\ref{eq:parameterspace}) is the projective space; i.e. $ V={\mathbb P}^{\delta}$.

In this situation, the family of CGC defined above, $\C=\{\C_\lambda\}_{\lambda\in\Pmat^{\delta}}$, is of type $[n,1,\delta]$ and, furthermore,  the generator matrix, ${\mathcal G}$ is canonical.

Under these assumptions, Theorem 1.11 of \cite{DMS:11} can be
strengthened as follows.

\begin{lemma}\label{lemma:constantGi}
The above-defined code $\C_\lambda$ (i.e. constructed over points $p_i=a_i z+b$ for $a_i$ non-zero and pairwise different)  is MDS if and only if $s_\lambda^{(j)}(b)\neq0$ for $0\leq j\leq\delta$.
\end{lemma}

\begin{proof}
Observe that the code has a generator matrix  $\mathcal{G}(z)=G_0+G_1z+\cdots+G_{\delta}z^{\delta}$, with
$G_i$ as in (\ref{eq:caract}) and $b_i=b$ for all $i=1,\ldots,n$.

Let us assume that ${\mathcal C}_{\lambda}$ is
MDS. Then, the linear block codes $G_j$ are MDS for all $j$. Otherwise,
if one of them, say $G_j$, is not MDS, there exists a degree
0 information word that is encoded by it into a codeword of
weight $<n$. Consequently, such an information word is encoded into a
convolutional codeword of weight strictly lower than $n(\delta+1)$,
contradicting the hypothesis of $\mathcal C$ being MDS (see
equation~(\ref{eq:singletonk=1})).

Note that the linear block code $G_j$ is MDS if and only if
its Hamming distance is $n$; that is, all entries in $G_j$
are non-zero. Recalling the explicit expression of these entries (see (equation~\ref{eq:caract})), one concludes.

Let us prove the converse. Note that the linear codes
 ${\small \begin{pmatrix}  G_u\\
 \vdots\\
 G_v  \end{pmatrix}}$ are MDS since all their maximal order minors have the form
$$
\begin{vmatrix}
    a_i^u&\dots&a_m^u\\ &\ddots&\\a_i^v&\dots&a_m^v
\end{vmatrix}
s_\lambda^{(u)}(b)\cdot\dots\cdot s_\lambda^{(v)}(b)
$$
and are non-zero by the assumptions. By \cite[Theorem 1.11]{DMS:11}, the CGC is of type $[n,1,\delta]$ and MDS, and one concludes.
\end{proof}

\begin{theorem}\label{th:equalb}
Let $\C=\{\C_\lambda\}_{\lambda\in\Pmat^{\delta}}$ be the family
of CGC of type $[n,1,\delta]$ defined above.

Then, the subset of $\Pmat^{\delta}$ consisting of the points
associated with MDS CGC's is a non-empty open subset. More precisely,
we have:
    $$
    \{\lambda\in\Pmat^\delta\,|\, \C_\lambda \text{ is MDS } \}\,=\,
        \{\lambda\in \Pmat^{\delta} \,|\, s_\lambda^{(j)}(b)\neq0\,,
        \forall
        0\leq j\leq\delta\}\,.
    $$

\end{theorem}

\begin{proof}
By  the previous Lemma, it only remains to prove that the set in the statement is non-empty. This follows easily because this set can be identified with:
$$
\psi^{-1}\Big(\big\{x=(x_0,\dots,x_{\delta})\in\F^{\delta+1} \,|\,
x_i\neq0, \text{ for $0\leq i\leq\delta$}\big\}\Big)
$$
where $\psi$ is the $\F$-linear map:
\begin{align*}
\F^{\delta+1}&\xrightarrow{\psi}\F^{\delta+1}\\
\lambda&\mapsto(s_\lambda^{(0)}(b),s_\lambda^{(1)}(b),\dots,s_\lambda^{(\delta)}(b))
\, ,
\end{align*}
which is the isomorphism whose associated matrix is ${\tiny\begin{pmatrix}
    1&*&\dots&*\\
    0&1&\dots&*\\
    \vdots &\ddots &\ddots&\vdots\\
    0& \dots &0 &1
\end{pmatrix}}$.
\end{proof}


\begin{remark}
If $b=0$, the equations of the non-MDS codes are $\lambda_i=0$ for
$0\leq i\leq\delta$. Thus, the code corresponding to $\lambda=(1,\dots,1)\in\Pmat^\delta $ is MDS.
If, moreover,  $a_j=a^{j-1}$ where $a$ is an element of $\F$ with
$\ord(a)\geq n$, we obtain the class of $1$-dimensional MDS
convolutional codes of type $[n,1,\delta]$ with generator matrix $\,
{\mathcal G}(z)=\sum_{i=0}^{\delta}z^i\begin{pmatrix} 1&a^i&\dots
&a^{(n-1)i}\end{pmatrix}$, constructed  by
Gluesing and Langfeld \cite{GL:06}.
\end{remark}


\begin{example}
Over the field extension $\F_4\simeq \F_2[\alpha]$ with $\alpha$ a
root of $x^2+x+1$ let us fix pairwise different points $p_1=z+\alpha, p_2=\alpha
z+\alpha,p_3=\alpha^2 z+\alpha\in\Pmat^1$ and let
$s_{\lambda}(t)=1+\alpha t+t^2$, where $\lambda=(1,\alpha,1)$. We
obtain the CGC of type $[n=3,1,\delta=2]$ generated by the canonical generator matrix
$$
{\mathcal G}(z)= \begin{pmatrix}
   z^2+\alpha z + 1 & \alpha^2z^2+\alpha^2 z + 1 & \alpha z^2+z + 1
\end{pmatrix}
\, .
$$

We have that $s^{(0)}_\lambda(\alpha)=1$,
$s^{(1)}_\lambda(\alpha)=\alpha$, $s^{(2)}_\lambda(\alpha)=1$.
Hence, from Theorem~\ref{th:equalb} the resulting code is MDS.

\end{example}
\begin{example}
Let us consider the projective line over $\F_5$, and let us fix
the points $p_1=z+1$, $p_2=2z+1$, $p_3=3z+1$,
$p_4=4z+1$, and $G=3p_\infty$. Any CGC of type $[n=4,1,\delta=3]$ with $s(t)=\lambda_0+\lambda_1t+
\lambda_2t^2+\lambda_3t^3 \in L(G)$ is generated by a
matrix ${\mathcal G}(z)=G_0+zG_1+z^2G_2+z^3G_3$, where $G_i$ are the matrices given by:
\begin{eqnarray*}
    & & G_0= (\lambda_0+\lambda_1  + \lambda_2 + \lambda_3) \begin{pmatrix}1&1&1&1\end{pmatrix}\\
    & & G_1= (\lambda_1+2\lambda_2+3\lambda_3)  \begin{pmatrix}1&2&3&4\end{pmatrix}\\
    & & G_2= (\lambda_2+3\lambda_3)
        \begin{pmatrix}1&4&4&1\end{pmatrix}\\
    & & G_3= \lambda_3 \begin{pmatrix}1&3&2&4\end{pmatrix}
    \, .
\end{eqnarray*}

According to Theorem \ref{th:equalb}, the necessary and sufficient
conditions for the above CGC to be MDS are
$$ 
\begin{array}{l}
    \lambda_0+\lambda_1+ \lambda_2 + \lambda_3 \neq 0\\
    \lambda_1+2\lambda_2+3\lambda_3\neq 0\\
    \lambda_2+3\lambda_3\neq 0\\
    \lambda_3 \neq 0
\end{array}
$$ 

\end{example}

\begin{example}\label{ex:7points}
Let us now consider the projective line over $\F_8\simeq
\F_2[\alpha]$ with $\alpha$ a root of $x^3+x+1$, and let us fix points
$\{p_i=a_iz+b\}_{1\leq i\leq 7}$, as well as the divisor
$G=2p_\infty$ and the rational function  $s(t)=
\lambda_0+\lambda_1t+ \lambda_2t^2\in L(G)$ with $\lambda_2\neq 0$.

Thus, by Theorem \ref{th:equalb}, the above CGC, which is of type $[n=7,1,\delta=2]$, is MDS if and only if
$$ 
\begin{array}{l}
    \lambda_0+\lambda_1 b+ \lambda_2 b^2\neq 0\\
    \lambda_1\neq 0
    \, .
\end{array}
$$ 

For instance, for $p_i=\alpha^{i-1}
z+\alpha$, $i=1,\ldots,7$,  the code, which is a CGC of type $[n=7,1,\delta=2]$, is MDS if and only if
\begin{equation}\label{eq:example418}
\begin{array}{l}
    \lambda_0+\lambda_1 \alpha+ \lambda_2 \alpha^2\neq 0\\
    \lambda_1\neq 0
\end{array}
\end{equation}

In particular, the CGC generated by $G_0+zG_1+z^2G_2$
where:
\begin{eqnarray*}
    & & G_0= \begin{pmatrix}1&1&1&1&1&1&1\end{pmatrix}\\
    & & G_1= \begin{pmatrix}\alpha^2&\alpha^3&\alpha^4&\alpha^5&\alpha^6&1&\alpha\end{pmatrix}\\
    & & G_2= \begin{pmatrix}\alpha^2&\alpha^4&\alpha^6&\alpha&\alpha^3&\alpha^5&1\end{pmatrix}
\end{eqnarray*}
is MDS, since it is the code associated with
$\lambda_0=\lambda_1=\lambda_2=\alpha^2$.
\end{example}

\subsection{Increasing the length of an MDS CGC}\label{subsec:increasinglength}

This subsection offers a constructive method for increasing the length of a given CGC. Using the results of \S\ref{subsec:l(C)} and \S\ref{subsec:enlargealphabet}, we show that it is possible to increase the length of the codes, preserving the property of being MDS.

More precisely,  given a CGC of type $[n,1,2]$ with canonical generator matrix:
     $$
    {\mathcal G}(z) \, := \, \big({\mathcal G}_1(z), \ldots, {\mathcal
    G}_n(z)\big)
    $$
we shall add another point to the divisor $p_1+\dots+p_n$ in order to obtain a CGC of type $[n+1,1,2]$.

Let $G=2p_{\infty}$, $p_i$ be the point
with affine coordinate $a_iz+b_i$ for $i=1,\ldots,n$ where we assume
that $a_i\neq 0$ and $\delta=2<n$. Let $s(t)$ be the rational
function  $\lambda_0+\lambda_1t +
\lambda_2t^2\in L(G)$. Let $p_{n+1}=a z+b$ be the new point we wish to add. Let ${\mathcal
F}(z)$ be the polynomial obtained by evaluating $s$ at $p_{n+1}$;
that is:
    \begin{equation}\label{eq:F(z)=s(p)}
    \begin{aligned}
    {\mathcal F}(z)\, & = \, F_0+F_1z+F_2z^2 \, = \,  s(p_{n+1})
    \, = \\ & = \,
    \sum_{j=0}^2 \lambda_j (a z+b)^j\, = \, \sum_{j=0}^2 s_{\lambda}^{(j)}(b)a^j z^j
    \end{aligned}
    \, .
    \end{equation}
We now consider $\widetilde{\mathcal C}$, a convolutional code of
dimension $1$ and length $n+1$, defined by the following generator
matrix:
    $$
    \widetilde{\mathcal G}(z) \, := \, \big({\mathcal G}_1(z), \ldots, {\mathcal
    G}_n(z), {\mathcal F}(z)\big)
    \, .
    $$

Let us introduce ${\mathcal F}_k$ as the $(k+1)\times (k+3)$-matrix
given by:
    \begin{equation}\label{eq:Fk=slidding}
    {\mathcal F}_k\,=\,
    \begin{pmatrix}
        F_{0} & F_{1} & F_{2} & 0 & \dots &  0 \\
        0 & F_{0} & F_{1} & F_{2} & \ddots  &  \vdots
        \\
        \vdots &  \ddots & \ddots & & \ddots  &  0 \\
        0 & \dots & 0 & F_{0} & F_{1} & F_{2}
        \end{pmatrix}
        \, .
    \end{equation}
If no confusion arises and ${\mathcal F}_k$ is of maximal rank, we shall denote by ${\mathcal F}_k$ the linear code generated by the image
of ${\mathcal F}_k:\F^{k+1}\to \F^{k+3}$, where we are identifying
$\F^{k+1}\simeq \langle 1,\ldots, z^k \rangle$ and $\F^{k+3}\simeq \langle 1,\ldots,
z^{k+2} \rangle$. With these notations, we have the following:

\begin{lemma}\label{lem:dtildeGgeqdG+dF}
Let $k_0:=\operatorname{max}\{\operatorname{l}(\widetilde{\mathcal C}) , \operatorname{l}({\mathcal C})\}$, where $\operatorname{l}({\mathcal C})$ is given by Equation \eqref{eq:lambdaG=integerpart}. If ${\mathcal F}_{k_0}$ is
of maximal rank, it holds that:
    $$
    \df(\widetilde{\mathcal C}) \,\geq\,  \df( {\mathcal C}) \;+ \;
        \operatorname{d} ({\mathcal F}_k) \qquad \forall k\geq k_0
    $$
\end{lemma}

\begin{proof}
Let us introduce matrices $G_i$ and $\widetilde G_i$. For $i=0,1,2$ we use the
polynomial decompositions of ${\mathcal G}(z)$ and $\widetilde{\mathcal C}$ as defining relations; that is:
    $$
    \begin{gathered}
    {\mathcal G}(z)=G_0+G_1z+G_2 z^2
    \\
    \widetilde{\mathcal G}(z)\,=\,
    \widetilde G_0+ \widetilde  G_1z+\widetilde G_2 z^2
    \end{gathered}
    \, ,
    $$
while we set $G_i=F_i=0$ for all $i<0$ and $i>2$.
Observe that $\widetilde G_i$ is precisely the juxtaposition of the
matrices $G_i$ and $F_i$, which we denote by $(G_i\vert F_i)$.

A degree $k$ information word, $\alpha(z)=\alpha_0+\alpha_1
z+\ldots+\alpha_k z^k$, is encoded as:
    {\small $$
    \alpha(z) \widetilde{\mathcal G}(z)
    \,=\,
    \sum_{j=0}^{k+2}
    \begin{pmatrix}
    \alpha_0 & \dots & \alpha_k
    \end{pmatrix}
    \begin{pmatrix}
    G_{j} & \vert & F_j \\
    G_{j-1} & \vert & F_{j-1} \\
    \vdots & & \vdots \\
    G_{j-k} & \vert & F_{j-k}
    \end{pmatrix}
   \cdot z^j
    $$}
or, equivalently, the juxtaposition of the polynomial words $\alpha(z){\mathcal G}(z)$ and:
    {\small
    $$ 
    \sum_{j=0}^{k+2}
    \begin{pmatrix}
    \alpha_0 & \dots & \alpha_k
    \end{pmatrix}
    \begin{pmatrix}
    F_j \\
    F_{j-1} \\
    \vdots  \\
    F_{j-k}
    \end{pmatrix}
    \cdot z^j
    \,=
    \,
    \begin{pmatrix}
    \alpha_0 & \dots & \alpha_k
    \end{pmatrix}
    {\mathcal F}_k
    \begin{pmatrix}
    z^0 \\
    z^1  \\
    \vdots  \\
    z^{k+2}
    \end{pmatrix}
    \, ,
    $$ 
    }
where ${\mathcal F}_k$ is the matrix given in equation~(\ref{eq:Fk=slidding}). Thus, the weight of $\alpha(z) \widetilde{\mathcal G}(z)$ can be computed as follows:
    $$
    \operatorname{w}(\alpha(z) \widetilde{\mathcal G}(z)) \,=\,
    \operatorname{w}((\alpha(z) {\mathcal G}(z))
    \;+\;
    \operatorname{w}(
        (\alpha_0 , \dots , \alpha_k) \cdot
        {\mathcal F}_k
        )
    $$

Considering the distances as lower bounds of the terms on the right
hand side and bearing in mind Theorem~\ref{th:l(G)},  we obtain:
    $$
    \operatorname{w}(\alpha(z) \widetilde{\mathcal G}(z))
    \,\geq \,
    \df ({\mathcal C})
    \;+\;
    \operatorname{min}\{ \operatorname{d}({\mathcal F}_k) \vert k\geq 0\}
    $$
for all $\alpha\neq 0$ and $ k\gg  0$.  Noting that $\{
\operatorname{d}({\mathcal F}_k) \vert k\geq 0\}$ is a non-increasing
sequence and allowing $k$ to be larger than or equal to $k_0$, the
statement follows.
\end{proof}

\begin{example}
Let us fix $\F_8\simeq \F_2[\alpha]$, where $\alpha$ is a root of $x^3+x+1$, as the
base field.

Let us consider the  CGC corresponding to the point
$\lambda =(1,1,1)\in \Pmat^2$, i.e the rational function $s(t)=1+t+t^2$, and
the points $p_1=z+\alpha$, $p_2=\alpha z+\alpha$, $p_3= \alpha^2 z
+\alpha$.

We thus have a CGC of type $[n=3,1,\delta=2]$ generated by:
$$
{\mathcal G}(z)=\left(
\begin{array}{ccc}
 z^2+z+\alpha^5 &
 \alpha^2 z^2+\alpha z +\alpha^5&
 \alpha^4 z^2+\alpha^2 z +\alpha^5
\end{array}
\right)
\, ,
$$
which has free distance 9 and is therefore MDS.

Let us now consider $\tilde{\mathcal C}$, the CGC  constructed from
${\mathcal G}(z)$ and the point  $p_4=\alpha^2 z +\alpha^4$ by the above procedure; that is, the code generated by:
$$
\tilde{\mathcal{G}}(z) = \left(
\begin{array}{c}
 z^2+z+\alpha^5 \\
 \alpha^2 z^2+\alpha z +\alpha^5 \\
 \alpha^4 z^2+\alpha^2 z +\alpha^5 \\
 \alpha^4 z^2+\alpha^2 z+\alpha^6
\end{array}
\right)^T
$$
We have that
$s(p_4)=\alpha^4 z^2+\alpha^2 z+\alpha^6$, i. e., $F_0=\alpha^6,
F_1=\alpha^2, F_2=\alpha^4$. Bearing in mind  equation~(\ref{eq:lambdaG=integerpart}) one has that $\operatorname{l}({\mathcal C})=6$ and
$\operatorname{l}(\widetilde{\mathcal C})=3$ and, also one easily checks that $d(\mathcal{F}_6)\geq 3$. Therefore,  the Lemma implies that $\widetilde{\mathcal C}$ is MDS and has free distance 12.
\end{example}

We are now ready to offer a way for increasing the length of a CGC.

Let us introduce the following notation. Let $h_k$ be the $k$-th complete symmetric function on the roots of ${\mathcal F}(z)=0$.

\begin{theorem}
With the previous notations, let $\mathcal C$ be an MDS CGC of type $[n,1,2]$.

Assume that  $\operatorname{d}{\tiny \begin{pmatrix} G_{2} \\  G_{1}  \\  G_{0} \end{pmatrix}}\geq 3$; $a\neq 0$; $s_{\lambda}^{(0)}(b)\neq 0$; $\lambda_2\neq 0$; and, $h_k\neq 0$ for all  $k\leq \operatorname{l}({\mathcal C})$.

It then holds that: i) $\widetilde{\mathcal C}$ is an MDS CGC of type $[n+1,1,\delta=2]$; ii)
    $\operatorname{d}{\tiny
            \begin{pmatrix}
            G_{2} & \vert & F_2 \\
            G_{1} & \vert & F_{1} \\
            G_{0} & \vert & F_{0}
            \end{pmatrix}}
            \geq 3$; and iii) $\operatorname{l}(\widetilde{\mathcal C})\leq
\operatorname{l}({\mathcal C})+1$.
\end{theorem}

\begin{proof}
First, note that ii) follows from the inequalities:
    $$
        \operatorname{d}{\small
            \begin{pmatrix}
            G_{2} & \vert & F_2 \\
            G_{1} & \vert & F_{1} \\
            G_{0} & \vert & F_{0}
            \end{pmatrix}}
	\,\geq \,
        \operatorname{d}{\small
            \begin{pmatrix}
            G_{2} \\
            G_{1}  \\
            G_{0}
            \end{pmatrix}}
    \,\geq \, 3
    $$

Let us now prove iii).  Recall that  ${\mathcal C}$ is of
type $[n,1,2]$ and that $\operatorname{l}({\mathcal C})$ is given by
formula (\ref{eq:lambdaG=integerpart}). Similarly, $\widetilde{\mathcal C}$ is of type $[n+1,1,2]$ and in
order to compute $\operatorname{l}(\widetilde{\mathcal C})$ by
formula (\ref{eq:lambdaG=integerpart}) we must juxtapose a new column given by the coefficients of
${\mathcal F}(z)$  to the
matrix ${\tiny
            \begin{pmatrix}
            G_{2} \\
            G_{1}  \\
            G_{0}
            \end{pmatrix}}$. Part ii) yields:
     $$
    \begin{gathered}
    \operatorname{l}(\widetilde{\mathcal C})
    \,\leq\,
       \left\lfloor 3(n+1)
        \operatorname{d}{\tiny
            \begin{pmatrix}
            G_{2} \\
            G_{1}  \\
            G_{0}
            \end{pmatrix}}^{-1}\right\rfloor
        \; -3
        \\
	    \operatorname{l}({\mathcal C})
    \,=\,
       \left\lfloor 3n \operatorname{d}{\tiny
            \begin{pmatrix}
            G_{2} \\
            G_{1}  \\
            G_{0}
            \end{pmatrix}}^{-1}\right\rfloor\; -3
            \end{gathered}
    $$
We have to show that $\operatorname{l}(\widetilde{\mathcal C})- \operatorname{l}({\mathcal C})\leq 1$. The case $\operatorname{d}{\tiny \begin{pmatrix} G_{2} \\  G_{1}  \\  G_{0} \end{pmatrix}}= 3$ is easy. For $\operatorname{d}{\tiny \begin{pmatrix} G_{2} \\  G_{1}  \\  G_{0} \end{pmatrix}}\geq 4$, it is enough to note that:
   $$
    \begin{gathered}
    \operatorname{l}(\widetilde{\mathcal C})
    \,\leq\,
       3(n+1)
        \operatorname{d}{\tiny
            \begin{pmatrix}
            G_{2} \\
            G_{1}  \\
            G_{0}
            \end{pmatrix}}^{-1}
        \; -3
        \\
	 \operatorname{l}({\mathcal C})
	 \geq  3 n \operatorname{d}{\tiny
            \begin{pmatrix}
            G_{2} \\
            G_{1}  \\
            G_{0}
            \end{pmatrix}}^{-1}
            \, -4
             \end{gathered}
   $$
and that $\lfloor 3\operatorname{d}{\tiny
            \begin{pmatrix}
            G_{2} \\
            G_{1}  \\
            G_{0}
            \end{pmatrix}}^{-1}\rfloor=0$.

Finally, we prove the first item. Recalling the Singleton bounds for
block codes and for convolutional codes and
Lemma~\ref{lem:dtildeGgeqdG+dF}, it will suffice to show that
${\mathcal F}_k$ is an MDS linear code for $
k=\operatorname{l}({\mathcal C})$ (see Theorem~\ref{th:l(G)});
that is, none of its maximal minors vanish.


Let us denote by $d_k$ the determinant of the
$(k+1)\times(k+1)$-matrix obtained by removing the first and last
column of ${\mathcal F}_k$:
    {
    $$
    d_k\,=\,
    \begin{vmatrix}
        F_{1} & F_{2} & 0 & \dots  & 0
        \\
        F_{0} & F_{1} & \ddots  &  & \vdots
        \\
        0 &  \ddots &  & \ddots & 0
        \\
        \vdots &  \ddots  & \ddots  & &   F_{2}
        \\
        0 & \dots & 0 &  F_{0} & F_{1}
        \end{vmatrix}
        \, .
    $$}
Bearing in mind equation~(\ref{eq:F(z)=s(p)}), we note that $F_0=
s_{\lambda}^{(0)}(b)$ and that $F_2=a^2\lambda_2$, and therefore
the minors of ${\mathcal F}_k$ are given by the following
expressions:
    { $$
    F_0^{i-1} \cdot d_{j-i-1}   \cdot F_2^{k+2-j}
    \,=\,
    \big(s_{\lambda}^{(0)}(b)\big)^{i-1} \cdot d_{j-i-1}   \cdot  \big(a^2 \lambda_2\big)^{k+2-j}
    $$}
for $1\leq i <j\leq k+2$. It will suffice to show that $d_k\neq 0$
for all $k\leq \operatorname{l}({\mathcal C})+1$. Using the
following recursion relation:
    $$
    d_k\,=\, F_1 d_{k-1} - F_0 F_2 d_{k-2}
    $$
one has that $d_k=(-F_2)^{k+1}\sum_{i=0}^{k+1} r_1^i
r_2^{k+1-i} = (- a^2\lambda_2)^{k+1} h_{k+1}$, where $h_{k+1}$ is the $(k+1)$-th complete symmetric function on the roots of ${\mathcal F}(z)=0$.
\end{proof}

\begin{remark}
Let us comment on the $(k+1)$-th complete symmetric
function. First, note that, $h_{k+1}$ is explicitly given by
$\sum_{i=0}^{k+1} r_1^i r_2^{k+1-i}$, where $r_1,r_2$ are  the
roots of ${\mathcal F}(z)=0$, and, since it is symmetric in
$r_1,r_2$, it belongs to the base field even if $r_1,r_2$ do not.
Further, $h_{k+1}$ can be expressed in terms of the coefficients
of ${\mathcal F}(z)$ by the recursion relation $h_{k+1}=-h_{k}
F_1/F_2 - h_{k-1} F_0/F_2$ with initial conditions $h_0=1$,
$h_1=F_1/F_2$.
\end{remark}

\begin{example}
Let us consider  Example \ref{ex:7points}. In particular, let
$\mathcal{G}(z)=G_0+zG_1+z^2 G_2$ be the generator matrix of the code
CGC of type $[n=7,1,\delta=2]$, where $D=\sum_i^7 p_i$,
$\,p_i=\alpha^{i-1} z +\alpha$ and
$s(t)=\lambda_0+\lambda_1t+\lambda_2t^2$ is the rational function
verifying (\ref{eq:example418}). We wish to extend this code to a
MDS CGC of type $[8,1,2]$ by adding a point
$p_8=az+b\notin\{p_i\}_{i=1}^7$, with $a\neq0$, to $D$.

Now, let $\mathcal{F}(z)$ be computed from
equation~(\ref{eq:F(z)=s(p)}) and let
$\widetilde{\mathcal{G}}=(\mathcal{G}\,|\,\mathcal{F})$. Note that
$\operatorname{d}{\tiny
\begin{pmatrix}
    G_{2} \\
    G_{1}  \\
    G_{0}
\end{pmatrix}}= 5$ while
$s_{\lambda}^{(0)}(b)\neq 0$ and $\lambda_2\neq 0$. Hence,  $\widetilde{\mathcal{G}}$ generates an MDS code if
$h_k\neq 0$ for all $k\leq l(\mathcal C)$. In fact
$l(\widetilde{\mathcal C})=l(\mathcal C)=1$. Thus, the condition that
we must check is equivalent to $d_1\neq0$, and
    {$$
    \begin{aligned}
    d_1\,& = \,
    \begin{vmatrix}
        F_1 & F_2 \\
        F_0 & F_1
    \end{vmatrix}\,=\,
    (\lambda_1 a)^2 + \lambda_2a^2(\lambda_2b^2+\lambda_1b+\lambda_0)
    \,=\\ & =\,
    \,a^2(\lambda_1^2 + \lambda_2^2b^2+\lambda_2\lambda_1b+\lambda_2\lambda_0)
    \end{aligned}
    $$}
Accordingly, the desired condition is   $\lambda_1^2 +
\lambda_2^2b^2+\lambda_2\lambda_1b+\lambda_2\lambda_0 \neq 0$.
\end{example}

\begin{remark}
As a final remark, let us point  out
that Corollary~\ref{cor:MDSopen} also implies that, for a given rational function $s(t)$, the  subset of eligible points:
    $$
    \{ p=az+b \in \Pmat^1 \text{ such that $\widetilde{\mathcal C}$ is MDS}\}
    $$
is Zariski open in $\Pmat^1$; i.e. \emph{bad} points fulfill certain algebraic equations and thus we have plenty of choices for such $p$. For instance, let us consider Example~\ref{ex:7points} with $s(t)=1+t+t^2$; i.e.
$\lambda_0=\lambda_1=\lambda_2=1$. Thus, the  point $p_8=az+b$
satisfies the condition that $\widetilde{\mathcal C}$ is indeed MDS if and only if $a\neq 0$
and $1+(b^2+b+1)\neq 0$; or, equivalently, $a\neq 0$ and $b\neq
0,1$.
\end{remark}

\section{Conclusions and further research}\label{sec:conclusion}

This work addresses several significant results with
promising perspectives for future constructions of
MDS convolutional codes. The results are related to the study
of three properties of the free distance and  the systematic construction of
convolutional codes.

First, our study of the sequence of row distances for a class of codes has allowed us to bound the stage in which this sequence has achieved the free distance of the convolutional code. Thus, we have derived an explicit
method for the calculation of the free distance.

Second, for the case of convolutional codes depending on parameters, the free distance has been studied as a function
on the parameter space and it has been shown that it is lower semicontinuous. In particular, we conclude that the property of being MDS is an open condition; that is, the subset of the parameter space corresponding to non-MDS codes is defined by a finite number of algebraic equations.

The last property is that the free distance is preserved when the alphabet grows; i.e. it is invariant under extensions of the base field.

Finally, an illustration of our results is offered for the case of $1$-dimensional MDS convolutional Goppa Codes. We compute the algebraic equations of the non-MDS locus in a number of examples.
In particular, by exhibiting a number of
examples on small fields, the existence of MDS convolutional codes is proved without requiring us to enlarge the base field.
We finish with a procedure, given an MDS CGC, for producing a new MDS CGC of greater length.

As a continuation of our approach, we believe that some lines of research deserve further investigation. On
the one hand, it would be desirable to obtain more general results about the relation between the sequence of row distances and the free distance. On the other, it would be interesting to explore different families of convolutional codes and to study the corresponding MDS locus.

\section*{Acknowledgments}

The authors wish to express their gratitude to  J.M. Mu\~{n}oz Porras and the anonymous referees
for their valuable comments and suggestions that have improved our work.

\end{document}